\documentclass[a4paper]{amsart}
\usepackage{amssymb}
\usepackage{amsmath}
\usepackage{bm}
\usepackage[colorlinks,linkcolor=blue]{hyperref}
\usepackage{enumerate}
\usepackage{mathrsfs}
\usepackage[all]{xy}

\newcommand{\schc}{\mathcal{C}}

\newcommand{\sche}{\mathcal{E}}
\newcommand{\schg}{\mathcal{G}}
\newcommand{\schk}{\mathcal{K}}
\newcommand{\schp}{\mathcal{P}}

\newcommand{\schr}{\mathcal{R}}
\newcommand{\schs}{\mathcal{S}}
\newcommand{\scht}{\mathcal{T}}
\newcommand{\schx}{\mathcal{X}}

\newcommand{\sho}{\mathscr{O}}

\newcommand{\shu}{\mathscr{U}}

\DeclareMathOperator{\rad}{rad}
\DeclareMathOperator{\deriver}{R\!}

\DeclareMathOperator{\Cone}{Cone}

\newcommand{\Kum}{\textup{\textsf{K}}}
\newcommand{\Num}{\textup{\textsf{N}}}
\newcommand{\Stwo}{\mathfrak{S}_2}

\DeclareMathOperator{\Aut}{Aut}
\DeclareMathOperator{\End}{End}
\DeclareMathOperator{\Hom}{Hom}
\DeclareMathOperator{\Ext}{Ext}

\DeclareMathOperator{\id}{id}
\DeclareMathOperator{\Gr}{Gr}
\DeclareMathOperator{\Fix}{Fix}
\DeclareMathOperator{\image}{im}
\DeclareMathOperator{\rank}{rank}

\newcommand{\bz}{\mathbb{Z}}

\newcommand{\br}{\mathbb{R}}
\newcommand{\bc}{\mathbb{C}}

\newcommand{\bp}{\mathbb{P}}

\DeclareMathOperator{\HH}{H}
\DeclareMathOperator{\Hoch}{HH}
\newcommand{\ch}{\textup{\textsf{ch}}}
\newcommand{\td}{\textup{\textsf{td}}}

\DeclareMathOperator{\Stab}{Stab}
\newcommand{\GLp}{\operatorname{GL^+}}
\newcommand{\grp}{{\tilde{\GLp}}(2,\br)}

\newcommand{\defi}[1]{{\textbf{\emph{#1}}}}
\newcommand{\cate}[1]{{\textbf{\textup{#1}}}}
\newcommand{\funct}[1]{{\textup{\textbf{\textsf{#1}}}}}

\theoremstyle{plain}
\newtheorem{theorem}{Theorem}[section]
\newtheorem{corollary}[theorem]{Corollary}

\newtheorem{proposition}[theorem]{Proposition}

\theoremstyle{definition}
\newtheorem{definition}[theorem]{Definition}
\newtheorem{remark}[theorem]{Remark}
\newtheorem{example}[theorem]{Example}

\author{Ziqi Liu}
\title[Two indecomposable Enriques categories]
{Two indecomposable Enriques categories}
\address{Dipartimento di Matematica ``F. Enriques'', Università degli Studi di Milano, Via Cesare Saldini 50, 20133 Milano, Italy}
\email{ziqi.liu@unimi.it}



\thanks{The author is a member of the INdAM group GNSAGA (2025-2026) and is partly supported by the FIS 3 Starting Grant No. FIS-2024-04715.}

\begin{document}

\begin{abstract}
This note proves that the Kuznetsov component of a quartic double solid or a special Gushel--Mukai threefold does not have non-trivial admissible subcategories.
\end{abstract}

\maketitle
\thispagestyle{plain}

\section{Introduction}
The bounded derived categories of smooth projective varieties have attracted much research attention over the past three decades. This interest soon spreads to their abstract analogues, the noncommutative varieties \cite{Ku22,Or20,Per26}.

A basic problem is to decide whether a noncommutative variety is \emph{indecomposable} i.e.~contains no non-trivial admissible subcategories. It is known that Calabi-Yau categories are indecomposable (see \cite[Proposition 5.5]{Ku19}) and so are the bounded derived categories of some algebraic varieties \cite{Ca24,JL26,KO18,Lin26,LY24,Oka11,Oka23,Pir23,Pir26}.

This note studies the indecomposability for two families of noncommutative curves. They are Kuznetsov components of some Fano threefolds \cite{Ku09} and, at the same time, also Enriques categories over K3 surfaces \cite{Liu26a,Liu26b}.

The first family arises from quartic double solids. One recalls that a quartic double solid $Y$ is a double cover $Y\rightarrow\bp^3$ branched on a quartic K3 surface. Some interesting properties of these categories are studied in \cite{Liu26a,PPZ26}. 

\begin{definition}
Let $Y$ be a quartic double solid, then its \defi{Kuznetsov component} $\schk_Y$ is the triangulated subcategory
$$\schk_Y:=\langle\sho_Y,\sho_Y(1)\rangle^{\perp}\subset\cate{D}^b(Y)$$
where $\sho_Y(1)$ is the inverse image of $\sho_{\bp^3}(1)$.	
\end{definition}

The special Gushel--Mukai threefolds give the second family. A special Gushel--Mukai threefold $Y$ is a double cover $Y\rightarrow W$ of a smooth codimension three linear section $W\subset\Gr(2,5)$ branched over a smooth divisor $S$ in $|\sho_W(2)|$. The geometry of this family is similar to the first one, see \cite{KP18,Liu26a,PPZ26}. 

\begin{definition}
Let $Y$ be a special Gushel--Mukai threefold, then its \defi{Kuznetsov component} $\schg_Y$ is the triangulated subcategory
$$\schg_Y:=\langle\sho_Y,\shu_Y^\vee\rangle^{\perp}\subset\cate{D}^b(Y)$$
where $\shu_Y$ is the tautological vector bundle pulled back from that of $\Gr(2,5)$.
\end{definition}

The main result of this note answers a part of \cite[Remark 2.18]{KP18}.

\begin{theorem}[Theorem \ref{5_main}]\label{1_main}
The Kuznetsov component of a quartic double solid or a special Gushel--Mukai threefold contains no non-trivial admissible subcategories.
\end{theorem}

\begin{remark}\label{1_remark}
These examples are special cases of the Enriques categories over K3 surfaces constructed in \cite{Liu26b}. In Section \ref{SS_4.3}, we explain the construction and record sufficient conditions under which the argument for Theorem \ref{1_main} applies.
\end{remark}

The key technical point of Theorem \ref{1_main} is to relate possible non-trivial admissible subcategories to autoequivalences of K3 surfaces using spherical functors \cite{AL17}. 

\subsection*{Organization}
Section 2 recalls some necessary background. Section 3 provides some new observations. Section 4 develops a framework to detect admissible subcategories of Enriques categories over K3 surfaces using spherical functors. Section 5 shows Theorem \ref{1_main} and discusses Remark \ref{1_remark}.

\subsection*{Conventions}
This note works over the field $\bc$. A noncommutative variety is a smooth proper pretriangulated idempotent-complete dg category. Any subcategory is assumed to be strictly full.

\section{Preliminaries}
This section recalls some general background. Throughout the section, we will fix a non-trivial noncommutative variety $\scht$. Its homotopy category is a non-trivial triangulated category and is, by abuse of notation, also denoted by $\scht$.

The reader who is not familiar with the theory of dg categories can regard $\scht$ as the bounded derived category $\cate{D}^b(X)$ of a smooth projective variety $X$.

\subsection{Some invariants}
The Grothendieck group $\Kum(\scht)$ of $\scht$ is the Grothendieck group of its homotopy triangulated category.

The properness of $\scht$ implies that $\deriver\Hom(E,F)$ is a bounded complex of finite-dimensional vector spaces for any $E$ and $F$. In this case, the Euler form
$$\chi(E,F)=\sum(-1)^n\dim\Hom(E,F[n])$$
is well-defined. It descends to a bilinear pairing $\chi\colon\Kum(\scht)\times\Kum(\scht)\rightarrow\bz$. 

\begin{definition}
The \defi{numerical Grothendieck group} of $\scht$ is the quotient $$\Num(\scht):=\Kum(\scht)/\rad(\chi)$$
where $\rad(\chi):=\{v\in\Kum(\scht)\mid\chi(v,w)=0\text{ for any } w\in\Kum(\scht)\}$.
\end{definition}

\begin{proposition}[{{\cite[Theorem 1.2]{TVdB16}}}]\label{2_N(T)-discrete}
The numerical Grothendieck group $\Num(\scht)$ is a finitely generated free abelian group.
\end{proposition}

The $i$-th \emph{Hochschild cohomology} of $\scht$ is $\Hoch^i(\scht)=\Hom(\funct{Id}_{\scht},\funct{Id}_{\scht}[i])$. The space $\Hoch^i(\scht)$ is finite dimensional and $\Hoch^{\bullet}(\scht)$ is bounded, see e.g.~\cite[Lemma 2.6]{BP23}.

To describe the Hochschild homology, we need the notion of a Serre functor. 

\begin{definition}
An autoequivalence $\funct{S}$ of $\scht$ is called a \defi{Serre functor} if it induces a functorial isomorphism $\Hom(E,\funct{S}F)\cong\Hom(F,E)^\vee$ of complex vector spaces for any two objects $E$ and $F$ of $\scht$.
\end{definition}

According to \cite[Corollary 3.5]{BK90}, there exists a unique Serre functor $\funct{S}_{\scht}$ for the noncommutative variety $\scht$. In this case, the $i$-th \emph{Hochschild homology} of $\scht$ has the form $\Hoch_i(\scht)=\Hom(\funct{Id}_{\scht},\funct{S}_{\scht}[-i])$. The complex $\Hoch_{\bullet}(\scht)$ is also a bounded complex of finite-dimensional vector spaces, see e.g.~\cite[Example 5.6]{CT14}.

The Hochschild homology and cohomology can be related when there exists an isomorphism between the Serre functor and a shift.

\begin{definition}
Let $n$ be an integer, then a noncommutative variety $\scht$ is called $n$-\defi{Calabi-Yau} if $\Hoch^0(\scht)=\bc$ and there exists an isomorphism $\funct{S}_{\scht}\stackrel{\sim}{\rightarrow}[n]$.
\end{definition}

\begin{remark}
In this case, one has $\Hoch^i(\scht)\cong\Hoch_{n-i}(\scht)$ by definitions.
\end{remark}

\subsection{Admissible subcategories}
The definition of an admissible subcategory of a triangulated category was introduced in \cite{BK90}, which provides a natural way to slice triangulated categories. A modern exposition is \cite[Chapter 7]{HuCu}.

\begin{definition}
A triangulated subcategory $\scht_0\subset\scht$ is called \defi{admissible} if the inclusion functor has a right adjoint and a left adjoint. An admissible subcategory $\scht_0\subset\scht$ is said to be \defi{non-trivial} if $\scht_0$ is neither zero nor $\scht$.
\end{definition}

According to \cite[Section 3]{Or18}, an admissible subcategory $\scht_0\subset\scht$ corresponds to a noncommutative subvariety of $\scht$. Conversely, a dg subcategory $\scht_0$ of the noncommutative variety $\scht$ induces an admissible subcategory at the level of homotopy categories if it is a noncommutative subvariety, see \cite[Proposition 1.6]{Or20}.

\begin{definition}
A noncommutative variety is \defi{indecomposable} if its homotopy category does not have any non-trivial admissible subcategory.
\end{definition}

The right orthogonal subcategory
$$\scht_0^{\perp}=\{E\in\scht\,|\,\Hom(X,E)=0\textup{ for any }X\in\scht_0\}$$ 
of an admissible subcategory $\scht_0\subset\scht$ is also an admissible subcategory. Hence $\scht$ is decomposed into two pieces. It can be generalized into the following notion.

\begin{definition}
A \defi{semiorthogonal decomposition} of the triangulated category $\scht$ consists of admissible subcategories $\scht_1,\dots,\scht_n$ of $\scht$ satisfying
\begin{itemize}
	\item[(a)] $\Hom(\scht_i,\scht_j) = 0$ for $i > j$;
	\item[(b)] for each non-zero object $E\in\scht$ there exists a sequence
	of exact triangles
	\begin{displaymath}
		0= \xymatrix@R=2em{
			E_n\ar[r]&E_{n-1}\ar[r]\ar[d]&E_{n-2}\ar[r]\ar[d]&\cdots\ar[r]&E_1\ar[r]\ar[d]&E_0\ar[d]\\
			&A_n\ar@{-->}[ul]&A_{n-1}\ar@{-->}[ul]&\ar@{-->}[ul]\cdots&\cdots&A_1\ar@{-->}[ul]
		} =E
	\end{displaymath}
	with $A_i:=\Cone(E_i\rightarrow E_{i-1})\in\scht_i$.
\end{itemize}
In this case, we write $\scht=\langle\scht_1,\dots,\scht_n\rangle$. A semiorthogonal decomposition is said to be \defi{non-trivial} if one of the admissible subcategories $\scht_1,\dots,\scht_n$ is non-trivial.
\end{definition}

Semiorthogonal decompositions are compatible with many invariants, such as the numerical Grothendieck group and Hochschild homology.

\begin{proposition}[{{\cite[Exercise 7.1.2]{HuCu}}}]\label{2_Kum-decomposition}
Consider a semiorthogonal decomposition $$\scht=\langle\scht_1,\dots,\scht_n\rangle$$ 
then the inclusion $\scht_i\hookrightarrow\scht$ induces a direct sum $\Kum(\scht)=\Kum(\scht_1)\oplus\cdots\oplus\Kum(\scht_n)$.
\end{proposition}

\begin{corollary}\label{2_Num-decomposition}
Consider a semiorthogonal decomposition $$\scht=\langle\scht_1,\dots,\scht_n\rangle$$ 
then the inclusion $\scht_i\hookrightarrow\scht$ induces a direct sum $\Num(\scht)=\Num(\scht_1)\oplus\cdots\oplus\Num(\scht_n)$.

\begin{proof}
It follows directly from Proposition \ref{2_Kum-decomposition} and the definitions.
\end{proof}
\end{corollary}

\begin{proposition}[{{\cite[Example 5.6]{CT14}}}]\label{2_Hoch-decomposition}
Consider a semiorthogonal decomposition $$\scht=\langle\scht_1,\dots,\scht_n\rangle$$ 
then there exists a decomposition $\Hoch_{\bullet}(\scht)=\Hoch_{\bullet}(\scht_1)\oplus\cdots\oplus\Hoch_{\bullet}(\scht_n)$.	
\end{proposition}

A similar statement does not hold for the Hochschild cohomology. Instead, it is compatible with orthogonal decomposition.

\subsection{Spherical functors}
The key tool of this note is the theory of spherical functors introduced by Anno and Logvinenko, see \cite{Add16,AL17,Se18} for details. Also, we refer the reader to \cite[Section 2.4]{Ku19} and \cite[Section 2]{KP21} for an alternative treatment.

Let $\funct{F}\colon\scht_1\rightarrow\scht_2$ be a dg functor of noncommutative varieties with a right adjoint functor $\funct{R}$ and a left adjoint functor $\funct{L}$, then its \emph{twist} is defined by
$$\funct{T}_{\funct{F}}:=\Cone(\funct{F}\funct{R}\rightarrow\funct{Id}_{\scht_2})$$
where the map inside is the counit of this adjunction, and its \emph{cotwist} is defined by
$$\funct{C}_{\funct{F}}:=\Cone(\funct{Id}_{\scht_1}\stackrel{\eta}{\rightarrow}\funct{R}\funct{F})[-1]$$
where the map inside is the unit of this adjunction.

\begin{definition}\label{2_spherical-functor-definition}
Suppose that any two of the following conditions hold:
\begin{itemize}
	\item[(a)] one has an isomorphism $\funct{R}\simeq\funct{L}\circ\funct{T}_{\funct{F}}[-1]$;
	\item[(b)] one has an isomorphism $\funct{R}\simeq\funct{C}_{\funct{F}}\circ\funct{L}[1]$;	
	\item[(c)] the twist $\funct{T}_{\funct{F}}$ is an autoequivalence of $\scht_2$;
	\item[(d)] the cotwist $\funct{C}_{\funct{F}}$ is an autoequivalence of $\scht_1$;
\end{itemize}
then a dg functor $\funct{F}\colon\scht_1\rightarrow\scht_2$ of noncommutative varieties is called \defi{spherical}.
\end{definition}
\begin{remark}
In fact, if any two of these conditions hold, then the other two hold automatically \cite{AL17} (see also \cite{Add16,Mea21}).
\end{remark}

This notion generalizes the spherical objects, see e.g.~\cite[Chapter 8]{Hu}. 

\begin{definition}
An object $E\in\scht$ is called $d$-\defi{spherical} if it satisfies $$\deriver\Hom(E,E)=\bc[0]\oplus\bc[-d]$$ 
and $\funct{S}_{\scht}(E)\simeq E[d]$.
\end{definition}
\begin{remark}
A $d$-spherical object is simply called a \emph{spherical} object if the noncommutative variety $\scht$ is a $d$-Calabi-Yau category.
\end{remark}

The associated \emph{spherical twist} $\funct{T}_E$ of a $d$-spherical object $E$ is defined by 
$$F\mapsto \Cone(\deriver\Hom(E,F)\otimes E\rightarrow F)$$ 
where the map inside is the evaluation map.

\begin{example}[{{\cite[Example 3.5]{AL16}}}]
Let $X$ be a $d$-dimensional Calabi-Yau manifold and $E$ be a spherical object in $\cate{D}^b(X)$, then the functor
$$\funct{O}_E:\cate{D}^b(\bc)\rightarrow\cate{D}^b(X),\quad K\mapsto K\otimes E$$
is a spherical functor. The associated twist $\funct{T}_{\funct{O}_E}$ is exactly the spherical twist $\funct{T}_E$.
\end{example}

\subsection{Equivariant categories}
To apply spherical functors to our case, we need to use equivariant categories. Here we refer to \cite{BP23,Ela15,Plo07} for more details.

Let $G$ be a finite abelian group, and assume that $G$ acts on the noncommutative variety $\scht$. Then one has the \emph{equivariant category} $\scht_G$ with respect to the $G$-action whose objects and morphisms are respectively the $G$-equivariant objects and $G$-equivariant morphisms. The new category $\scht_G$ is also a noncommutative variety, see for example \cite[Proposition 3.15]{BP23} and \cite[Corollary 8.10]{Ela15}.

Moreover, there exists a natural action of the dual group $G^\vee:=\Hom(G,\bc^\times)$ on $\scht_G$ induced by a $G$-action on $\scht$. 

\begin{theorem}[{{\cite[Theorem 4.2]{Ela15}}}]\label{2_dual-dual-reconstruction}
There exists an equivalence $\funct{X}\colon(\scht_G)_{G^\vee}\cong\scht$ such that the $G$-action is identified with the $(G^\vee)^\vee$-action by $(G^\vee)^\vee\cong G$.
\end{theorem}

There are mutually left and right adjoint functors for the equivariant category construction: the \emph{forgetful functor} $\funct{Forg}\colon\scht_G\rightarrow\scht$ defined by forgetting the $G$-equivariant structures, and the \emph{inflation functor} $\funct{Inf}\colon\scht\rightarrow\scht_G$ given on objects by sending $A$ to $\bigoplus_{g \in G} g_*A$ endowed with the natural $G$-equivariant structure. 

\begin{proposition}[{{\cite[Lemma 2.11]{PPZ26}}}]\label{2_dual-dual-forgetful}
The forgetful functor $(\scht_G)_{G^\vee}\rightarrow\scht_G$ for the induced $G^\vee$-action on $\scht_G$ is isomorphic to the inflation functor $\scht\rightarrow\scht_G$ for the $G$-action on $\scht$ through the equivalence in Theorem \ref{2_dual-dual-reconstruction}.
\end{proposition}

\subsection{Enriques categories}
The main objects in this note are Enriques categories over K3 surfaces. It was introduced in \cite{BP23} and developed in \cite{Liu26a,PPZ26}. 

\begin{definition}
An \defi{Enriques category} $\sche$ is a noncommutative variety such that the autoequivalence $\amalg:=\funct{S}_{\sche}[-2]$ is an involution and generates an $\Stwo$-action on $\sche$ such that the equivariant category $\schc:=\sche_{\Stwo}$ is $2$-Calabi--Yau.
\end{definition}

\begin{remark}
Theorem \ref{2_dual-dual-reconstruction} provides an equivalence $\schc_{\Stwo^\vee}\simeq\sche$ and we denote a generator of $\Stwo^\vee$ by $\Pi$. Then $\Pi$ is an involution on the 2-Calabi-Yau category $\schc$.
\end{remark}

An Enriques category is said to be \emph{over a K3 surface} if the 2-Calabi-Yau category $\schc$ is equivalent to $\cate{D}^b(S)$ for some K3 surface $S$. 

\begin{example}
The bounded derived category $\cate{D}^b(T)$ of an Enriques surface $T$ is an Enriques category over a K3 surface. 
\end{example}

The example above has non-trivial admissible subcategories \cite{Zu97}, so we do not expect all Enriques categories to be indecomposable. Instead, we focus on those sharing some properties with smooth projective curves.

\begin{example}[\cite{KP17,PPZ26}]\label{2_EC-example1}
Let $Y$ be a quartic double solid and $S\subset\bp^3$ be the associated K3 surface, then the Kuznetsov component $\schk_Y$ is an Enriques category and the associated $2$-Calabi-Yau category is equivalent to $\cate{D}^b(S)$.
\end{example}

\begin{example}[\cite{KP17,PPZ26}]\label{2_EC-example2}
Let $Y$ be a special Gushel--Mukai threefold, then the associated smooth divisor $S$ in $W\subset\Gr(2,5)$ is a K3 surface and $\schg_Y$ is an Enriques category whose associated $2$-Calabi-Yau category is equivalent to $\cate{D}^b(S)$.
\end{example}

An important curve-like property of these two Enriques categories is that their numerical Grothendieck groups have rank two, see Proposition \ref{5_numerical}.

\subsection{Stability conditions}
The last ingredients are the autoequivalences on K3 surfaces. To exploit them, it is inevitable to involve stability conditions. 

The notion of a stability condition, introduced in \cite{Bri07}, is defined on the level of triangulated categories. In this note, it suffices to consider the \emph{numerical} ones. 

\begin{definition}\label{2_stab-definition}
A \defi{(numerical) stability condition} on $\scht$ is a pair $\sigma=(\schp,Z)$ where the \emph{central charge} $Z$ is a group homomorphism $Z\colon\Num(\scht)\rightarrow\bc$ and the \emph{slicing} $\schp$ is a collection of abelian subcategories $\schp(\phi)\subset\scht$ for all $\phi\in\br$ such that
\begin{itemize}
	\item[(a)] $\schp(\phi+1)=\schp(\phi)[1]$ and $\Hom(\schp(\phi_1),\schp(\phi_2))=0$ for any $\phi_1>\phi_2$;
	\item[(b)] one has $Z(E)\in\br_{>0}\cdot\exp(\pi\phi\bm{i})$ for any $0\neq E\in\schp(\phi)$;
	\item[(c)] every object $E$ of $\scht$ admits a Harder--Narasimhan filtration
	\begin{displaymath}
	 \xymatrix{
		0=E_0\ar[r]&E_1\ar[r]&E_2\ar[r]&\cdots\ar[r]&E_{n-1}\ar[r]&E_n=E
    }\end{displaymath}
	with $A_i:=\Cone(E_{i-1}\rightarrow E_i)\in\schp(\phi_i)$ for all $i$ such that $\phi_1>\cdots>\phi_n$. 
	\item[(d)] there exists a quadratic form $Q$ on $\Num(\scht)_{\br}$ such that $\ker(Z_{\br})$ is negative definite with respect to $Q$, and one has $Q([E])\geq0$ for any $E\in\schp(\phi)$.
\end{itemize}
The set of all such stability conditions on $\scht$ is denoted by $\Stab(\scht)$ or simply $\Stab(X)$ if one has $\scht=\cate{D}^b(X)$ for a smooth projective variety $X$.
\end{definition}

The set $\Stab(\scht)$ carries a complex manifold structure \cite{Bri07}. The universal cover 
$$
\grp  = \left \{ \tilde{g}=(M, f) ~  \middle\vert ~ 
\begin{array}{l} 
	M \in \GLp(2,\br), f \colon \br \rightarrow \br \text{ is} \\ 
	\text{an increasing function such} \\
	\text{that for all $\phi \in \br$ we have} \\
	f(\phi+1) = f(\phi) + 1 \textup{ and }\\
	M \cdot e^{\pi \phi\bm{i}} \in \br_{>0} \cdot e^{\pi f(\phi)\bm{i}} 
\end{array}  \right \}
$$
of the group $\GLp(2,\br)$ acts on $\Stab(\scht)$ on the right via
$$(\schp,Z).(M,f):=(\schp',M^{-1}\circ Z)$$
where $\schp'$ is defined by $\schp'(\phi):=\schp(f(\phi))$. Also, the group $\Aut(\scht)$ acts on $\Stab(\scht)$ on the left. These two group actions commute.

A stability condition $\sigma$ is called \emph{$\Phi$-fixed} for some element $\Phi\in\Aut(\scht)$ if one has $\Phi.\sigma=\sigma$ and is called \emph{$\Phi$-invariant} if $\Phi.\sigma=\sigma.\tilde{g}$ for some $\tilde{g}\in\grp$.

\begin{theorem}[{{\cite[Theorem A.10]{JLLZ24}}}]\label{2_Serre-invariant}
Consider an Enriques category $\sche$ described in Example \ref{2_EC-example1} or Example \ref{2_EC-example2}, then all the $\funct{S}_{\sche}$-invariant stability conditions on $\sche$ are contained in the same $\grp$-orbit of $\Stab(\sche)$. 
\end{theorem}

\subsection{K3 surfaces}
Let $S$ be a projective K3 surface, then there exists a canonical weight-two Hodge structure on the \emph{Mukai lattice}
$$\tilde{\HH}(S,\bz)=\HH^0(S,\bz)\oplus\HH^2(S,\bz)\oplus\HH^4(S,\bz)$$
equipped with the Mukai pairing $\langle(r_1,\alpha_1,s_1),(r_2,\alpha_2,s_2)\rangle:=\alpha_1.\alpha_2-r_1s_2-r_2s_1$ where the product $\alpha_1.\alpha_2$ is the usual intersection in $\HH^2(S,\bz)$.

The algebraic part of the Mukai lattice $\tilde{\HH}(S,\bz)$ is identified with $\Num(S)$ through the Mukai vector $v(-):=\ch(-).\sqrt{\td(S)}=\ch(-).(1,0,1)$. The sublattice $\Num(S)$ has signature $(2,\rho)$ under the Mukai pairing where $\rho$ is the Picard rank of $S$. 

The space $\Gr^+_2(\Num(S)_{\br})$ of oriented positive definite two-dimensional subspaces of $\Num(S)_{\br}$ has two connected components and an isometry on $\Num(S)$ is said to be \emph{orientation-preserving} if it preserves each of the two connected components.

\begin{theorem}[\cite{HMS09,Or97}]\label{2_K3-criterion-we-need}
An autoequivalence $\Phi$ on $\cate{D}^b(S)$ canonically induces a Hodge isometry $\Phi^H$ on $\tilde{\HH}(S,\bz)$ which is orientation-preserving on $\Num(S)$.
\end{theorem}

A stability condition $(\schp,Z)$ on $\cate{D}^b(S)$ is said to be \emph{geometric} if all the skyscraper sheaves are simple objects in a fixed $\schp(\phi)$. By virtue of \cite{Bri08}, the complex manifold $\Stab(S)$ has a non-empty distinguished connected component $\Stab^{\dag}(S)$ containing all the geometric stability conditions. The subgroup of $\Aut(\cate{D}^b(S))$ which preserves the connected component $\Stab^{\dag}(S)$ is denoted by $\Aut^{\dag}(\cate{D}^b(S))$.

In this note, we need to use the following observation.

\begin{proposition}\label{2_K3-property-Bridgeland-coro}
Consider an autoequivalence $\Pi$ on $\cate{D}^b(S)$ such that
$$\{\sigma\in\Stab(S)\mid\Pi.\sigma=\sigma\}=\sigma_0.\grp$$
for some stability condition $\sigma_0$ in $\Stab^{\dag}(S)$, then any autoequivalence $\Phi$ on $\cate{D}^b(S)$ satisfying $\Phi^H=\id$ and $\Phi\Pi=\Pi\Phi$ is isomorphic to an even shift.
	
\begin{proof}
Since $\Phi\Pi=\Pi\Phi$, one has $\Pi.(\Phi.\sigma_0)=\Phi.\sigma_0$ and then $\Phi.\sigma_0=\sigma_0.\tilde{g}$ for some element $\tilde{g}=(M,f)\in\grp$. The stability conditions $\sigma_0$ and $\Phi.\sigma_0$ have the same central charge $Z$ as $\Phi^H=\id$. It follows that $M$ is the identity matrix, because the image of $Z_{\br}$ has real rank two. Then one has $\Phi.\sigma_0=[2m].\sigma_0$ for some integer $m$ by construction. One concludes by applying \cite[Proposition 1.4]{Hu16} to $\Phi\circ[-2m]$.
\end{proof}
\end{proposition}

\section{Technical preparations}
This section provides some technical statements that are not directly available in the literature. We will still fix a noncommutative variety $\scht$.

\subsection{Three observations}
To begin with, we record three observations for the noncommutative variety $\scht$. The first one is about idempotents in $\Hoch^0(\scht)$.

\begin{proposition}\label{3_observation-1}
There exists a decomposition 
$$\scht\simeq\scht_1\oplus\cdots \oplus\scht_n$$ 
such that $\Hoch^0(\scht_i)$ has no non-trivial idempotents for every $i$.

\begin{proof}
Choose pairwise orthogonal primitive idempotents $e_1,\dots,e_n$ in $\Hoch^0(\scht)$ such that $e_1+\cdots +e_n=1$. Then one has $\funct{Id}_{\scht}\simeq\funct{F}_1\oplus\cdots\oplus\funct{F}_n$ for some non-trivial functors $\funct{F}_i$ in $\End(\scht)$. Let $\scht_i$ be the essential image of $\funct{F}_i$, then $e_i$ acts as the identity on $\scht_i$ and other idempotents act as zero. In this case, one has
$$\scht\simeq\scht_1\oplus\cdots \oplus\scht_n$$
and $\Hoch^0(\scht_i)=e_i\Hoch^0(\scht)$ by construction. Since $e_i$ is primitive, the Hochschild cohomology $\Hoch^0(\scht_i)$ has no non-trivial idempotents.
\end{proof}
\end{proposition}

The second observation is about the direct summands of $\scht$ in Proposition \ref{3_observation-1}.

\begin{proposition}\label{3_observation-2}
Suppose that $\Hoch^0(\scht)$ has no non-trivial idempotents. Then for any four autoequivalences $\funct{A},\funct{B},\funct{C},\funct{D}$ on $\scht$ satisfying $\funct{A}\oplus\funct{B}\simeq\funct{C}\oplus\funct{D}$, one has either $\funct{A}\simeq\funct{C}$ and $\funct{B}\simeq\funct{D}$ or $\funct{A}\simeq\funct{D}$ and $\funct{B}\simeq\funct{C}$.

\begin{proof}
Let $\funct{F}$ be an autoequivalence, then one has $$\Hom(\funct{F},\funct{F})=\Hom(\funct{Id}_{\scht},\funct{Id}_{\scht})=\Hoch^0(\scht)$$
where the first equality is given by composing with $\funct{F}^{-1}$. Since $\Hoch^{0}(\scht)$ is finite-dimensional and has no non-trivial idempotents, it is local. The assertion now follows from the Krull–Schmidt uniqueness theorem, see e.g.~\cite[Theorem 4.2]{Kra15}.
\end{proof}
\end{proposition}

The last observation needs the noncommutative Chern character
$$\ch\colon\Kum(\scht)\rightarrow\Hoch_0(\scht)$$
defined in \cite[Page 4]{Shk13}. It is compatible with the usual Chern character through the Hochschild--Kostant--Rosenberg isomorphism \cite{HKR62,Swan96} when $\scht$ is the bounded derived category $\cate{D}^b(X)$ of a smooth projective variety $X$.

\begin{proposition}\label{3_observation-3}
One has $\dim\Hoch_0(\scht)\geq\rank\Num(\scht)$.
	
\begin{proof}
Set $r = \rank\Num(\scht)$ and choose classes $\kappa_1,\dots,\kappa_r$ in $\Kum(\scht)$ whose numerical classes form a basis of $\Num(\scht)$. Suppose for contradiction that their Chern characters are linearly dependent in $\Hoch_0(\scht)$. Then there exist complex numbers $\lambda_1,\dots,\lambda_r$, not all zero, such that
$$\kappa=\sum\lambda_i\kappa_i\in\Kum(\scht)_\bc$$
and $\ch(\kappa)=0$. The image of $\kappa$ in $\Num(\scht)_{\bc}$ is nonzero because the numerical classes of $\kappa_i$ are linearly independent. However, one has 
$$\chi(\kappa,x)=\langle\ch(x),\ch(\kappa)^\vee\rangle=0$$
for every $x\in\Kum(\scht)_{\bc}$ using noncommutative Hirzebruch–Riemann–Roch \cite{Shk13} which forces the numerical class of $\kappa$ to vanish, impossible. Thus $\dim\Hoch_0(\scht)\geq r$.
\end{proof}
\end{proposition}

\subsection{Invariant stability conditions}
Stability conditions are compatible with group actions and equivariant categories. Let $\scht_G$ be the equivariant category for a finite abelian $G$-action on $\scht$ and assume that $\sigma=(\schp,Z)$ is a \emph{$G$-fixed} stability condition on $\scht$ i.e.~it is fixed by the autoequivalence induced by any $g\in G$.

In this case, we can define a pair $\sigma_G=(\schp_G,Z_G)$ on $\scht_G$ by declaring
\begin{align*}
	Z_G&:=Z\circ f\\
	\schp_G(\phi)&:=\{E\in\scht_G\mid\funct{Forg}(E)\in\schp(\phi)\}
\end{align*}
where $f\colon\Num(\scht_G)\rightarrow\Num(\scht)$ is the group homomorphism induced by the forgetful functor. The pair $\sigma_G=(\schp_G,Z_G)$ on $\scht_G$ satisfies the first three axioms in Definition \ref{2_stab-definition} by virtue of \cite[Theorem 4.9]{PPZ26}. The only issue is that one might not have a quadratic form on the whole space $\Num(\scht_G)_{\br}$ satisfying Definition \ref{2_stab-definition}(d).

In the following, we will provide two sufficient conditions under which one can find a desired quadratic form on $\Num(\scht_G)_{\br}$.

\begin{proposition}\label{3_invariant-stab-general}
The pair $\sigma_G=(\schp_G,Z_G)$ is a (numerical) stability condition on $\scht_G$ if the linear map $f_{\br}\colon\Num(\scht_G)_{\br}\rightarrow\Num(\scht)_{\br}$ is injective.
	
\begin{proof}
By definition, one can find a quadratic form $Q$ on $\Num(\scht)_{\br}$ such that $\ker(Z_{\br})$ is negative definite and $Q([E])\geq 0$ for any $E\in\schp(\phi)$. Set $Q_G:=Q\circ f_{\br}$, then one sees immediately that $Q_G$ is non-negative on numerical classes of objects in $\schp_G(\phi)$ by construction. The injectivity of $f_{\br}$ ensures that the restriction of $Q_G$ on $\ker((Z_G)_{\br})$ is negative definite. So the form $Q_G$ satisfies Definition \ref{2_stab-definition}(d).
\end{proof}
\end{proposition}

\begin{remark}\label{3_invariant-stab-general-rem}
The condition on $f_{\br}$ is satisfied if the endomorphism on $\Num(\scht_G)$ induced by any element in $G^\vee$ is the identity. In fact, this ensures that $i_{\br}\circ f_{\br}$ is a multiple of the identity map on $\Num(\scht_G)_{\br}$, where $i$ is the map induced by $\funct{Inf}$.
\end{remark}

The next criterion is related to Enriques categories.

\begin{proposition}\label{3_invariant-stab-Enriques}
The pair $\sigma_G=(\schp_G,Z_G)$ is a (numerical) stability condition on $\scht_G$ if $\scht_G$ is 2-Calabi--Yau and its Euler form is positive definite on $\ker(f_{\br})$.
	
\begin{proof}
Choose the quadratic form $Q$ on $\Num(\scht)_{\br}$ as in Definition \ref{2_stab-definition}(d). Since $Q$ is negative definite on $\ker(Z_{\br})$, the set
$$\{v\in\Num(\scht)_{\br}\mid Q(v)\geq0\textup{ and }|Z(v)|\leq 1\}$$
is compact, possibly up to a rescaling. One can find a universal $\delta>0$ such that $|Z(v)|\geq\delta$ for every non-zero class in $\Num(\scht)$ satisfying $Q(v)\geq0$, because $\Num(\scht)$ is a finite rank lattice (see Proposition \ref{2_N(T)-discrete}).

A non-zero object $E$ in $\schp_G(\phi)$ admits a finite Jordan--Hölder filtration in the abelian category $\schp_G(\phi)$. Let $E_1,\dots,E_n$ be the simple factors. Then one has
$$\chi_{\scht_G}(E_i,E_j)=\dim\Hom(E_i,E_j)+\dim\Hom(E_j,E_i)-\dim\Ext^1(E_i,E_j)\leq 2$$
using Serre duality. It follows that $\chi_{\scht_G}(E,E)\leq 2n^2\leq \frac{2}{\delta^2}|Z_G(E)|^2$.

Since $\chi_{\scht_G}|_{\ker(f_{\br})}$ is positive definite by assumption, one has
$$\ker((Z_G)_{\br})=\ker(f_{\br})\oplus \left(\ker((Z_G)_{\br})\cap\ker(f_{\br})^{\perp}\right)$$
and then, for large enough $\lambda>0$, the quadratic form $Q_G$ on $\Num(\scht_G)_{\br}$ defined by 
$$Q_G(v):=\lambda Q(f_{\br}(v))-\chi_{\scht_G}(v,v)+\frac{2}{\delta^2}|Z_G(v)|^2$$ 
is negative definite on $\ker(Z_G)_{\br}$. One also checks that $$Q_G([E])\geq\lambda Q(f([E]))\geq0$$
for every object $E$ in $\schp_G(\phi)$, so $Q_G$ is the desired form.
\end{proof}
\end{proposition}

The following example illustrates the applications of these two criteria. 

\begin{example}\label{3_invariant-stab-Enriques-Example}
Let $\sche$ be an Enriques category in Example \ref{2_EC-example1} or Example \ref{2_EC-example2} and let $\cate{D}^b(S)$ be the associated 2-Calabi--Yau category. Then the involution $\amalg:=\funct{S}_{\sche}[-2]$ induces the trivial action on $\Num(\sche)$. In this case, the forgetful functor 
$$\sche=\cate{D}^b(S)_{\Stwo^\vee}\rightarrow\cate{D}^b(S)$$ 
satisfies Proposition \ref{3_invariant-stab-general}. Moreover, one can see from \cite[Propositions 5.6 and 6.6]{Liu26a} that the Euler form of $\cate{D}^b(S)$ is positive definite on $\ker(f_{\br})$ for the homomorphism $f\colon\Num(S)\rightarrow\Num(\sche)$ induced by the forgetful functor $\cate{D}^b(S)=\sche_{\Stwo}\rightarrow\sche$.

In this case, any $\amalg$-fixed stability condition $\sigma$ on $\sche$ induces a stability condition on $\cate{D}^b(S)$ according to Proposition \ref{3_invariant-stab-Enriques}. Let $\Pi$ be the generator of the dual group action on $\cate{D}^b(S)$, then any $\Pi$-fixed stability condition on $\cate{D}^b(S)$ induces a stability condition on $\sche$ by Proposition \ref{3_invariant-stab-general}.
\end{example}

In the end, we would like to mention that $(\sigma_G)_{G^\vee}$ is essentially equal to the original stability condition $\sigma$ by the equivalence in Theorem \ref{2_dual-dual-reconstruction}.

\begin{proposition}[{{\cite[Lemma 4.12]{PPZ26}}}]\label{3_invariant-stab-go-back}
Suppose that $\sigma=(\schp,Z)$ is $G$-fixed and $\sigma_G$ is $G^\vee$-fixed, then $(\sigma_G)_{G^\vee}$ corresponds to $(\schp,|G|\cdot Z)$ through $\funct{X}\colon(\scht_G)_{G^\vee}\simeq\scht$.
\end{proposition}

\section{The general criteria}
This section develops a general framework for excluding some potential non-trivial admissible subcategories of certain noncommutative varieties.

\subsection{Quasi-Enriques categories}
The notion of a quasi-Enriques category is designed for relating the spherical functors and the Enriques categories.

\begin{definition}
A noncommutative variety $\sche$ is said to be a \defi{quasi-Enriques category} if there exists a 2-Calabi-Yau category $\schc$ and a functor 
$$\funct{F}\colon\sche\rightarrow\schc$$
with left adjoint $\funct{L}$ and right adjoint $\funct{R}$, such that there exists an isomorphism
$$\funct{R}\funct{F}\simeq\funct{Id}_{\sche}\oplus\funct{S}_{\sche}[-2]$$
under which the adjunction unit $\funct{Id}_{\sche}\rightarrow\funct{R}\funct{F}$ is the inclusion of the first summand.
	
A quasi-Enriques category is said to be \defi{over a K3 surface} if $\schc$ is equivalent to the bounded derived category $\cate{D}^b(S)$ of a K3 surface $S$.
\end{definition}

The property of being a quasi-Enriques category is inherited by its admissible subcategories. It is in fact a variant of \cite[Proposition 2.1]{Add16}.

\begin{proposition}\label{4_quasi-En-subcat}
Consider an admissible subcategory
$\funct{J}\colon\schr\hookrightarrow\sche$ of a quasi-Enriques category $\sche$ and the functor $\funct{F}\colon\sche\rightarrow\schc$ defining $\sche$, then $\funct{F}_0:=\funct{F}\circ \funct{J}$ realizes the noncommutative variety $\schr$ as a quasi-Enriques category.

\begin{proof}
Since $\schr$ is admissible and $\sche$ is a quasi-Enriques category, the composite $\funct{F}_0$ has a right adjoint $\funct{R}_0$ and a left adjoint. Its right adjoint has the form $\funct{K}\circ\funct{R}$ where $\funct{K}$ is the right adjoint of $\funct{J}$ and $\funct{R}$ is the right adjoint of $\funct{F}$. One has
$$\funct{R}_0\circ\funct{F}_0\simeq\funct{K}\circ(\funct{Id}_{\sche}\oplus\funct{S}_{\sche}[-2])\circ\funct{J}\simeq\funct{Id}_{\schr}\oplus\funct{S}_{\schr}[-2]$$
using the fact $\funct{S}_{\schr}=\funct{K}\circ\funct{S}_{\sche}\circ\funct{J}$. The adjunction unit $\funct{Id}_{\schr}\rightarrow\funct{R}_0\funct{F}_0$ is the inclusion of the first summand in this direct sum by construction. 
\end{proof}
\end{proposition}

The core property of a quasi-Enriques category $\sche$ is that the functor $\funct{F}$ in its definition is spherical.

\begin{proposition}\label{4_quasi-En-ppst}
The functor $\funct{F}\colon\sche\rightarrow\schc$ defining a quasi-Enriques category $\sche$ is spherical. Moreover, one has $\funct{C}_{\funct{F}}\simeq\funct{S}_{\sche}[-3]$ and $\funct{T}_{\funct{F}}\funct{F}\simeq\funct{F}\funct{S}_{\sche}[-1]$.

\begin{proof}
By definition, one has $\funct{C}_{\funct{F}}\simeq\funct{S}_{\sche}[-3]$ for the cotwist associated with $\funct{F}$. Then
$$\funct{R}\simeq\funct{S}_{\sche}\circ\funct{L}\circ\funct{S}_{\schc}^{-1}\simeq\funct{S}_{\sche}\circ\funct{L}[-2]\simeq\funct{C}_{\funct{F}}\circ\funct{L}[1]$$
using $\funct{S}_{\schc}\simeq[2]$ and Serre duality. Hence $\funct{F}$ is a spherical functor. The statement about spherical twist $\funct{T}_{\funct{F}}$ follows from a direct computation.
\end{proof}
\end{proposition}

This property allows us to make the following observation.

\begin{proposition}\label{4_quasi-En-rank1}
Suppose that $\sche$ satisfies the following conditions
\begin{itemize}
	\item[(a)] $\rank\Num(\sche)=1$;
	\item[(b)] a generator $v$ of $\Num(\sche)$ satisfies $\chi_{\sche}(v,v)<0$;
\end{itemize}
then $\sche$ cannot be a quasi-Enriques category over a K3 surface $S$.

\begin{proof}
Suppose for contradiction that $\sche$ is a quasi-Enriques category and choose a functor $\funct{F}\colon\sche\rightarrow\cate{D}^b(S)$ defining it. One has $(\funct{S}_{\sche})_*v=\epsilon v$ for some $\epsilon=\pm1$ by our assumptions, then Serre duality ensures that $\epsilon=1$ as $\chi_{\sche}(v,v)\neq0$.
	
Set $w:=\funct{F}_*v$ and let $\funct{R}$ be the right adjoint of $\funct{F}$, then 
$$\langle w,w\rangle=-\chi_S(w,w)=-\chi_{\sche}(v,(\funct{R}\funct{F})_*v)=-2\chi_{\sche}(v,v)$$
using adjunction and $(\funct{S}_{\sche})_*v= v$. 
	
Choose $x\in\Num(S)$, one writes $\funct{R}_*x = mv$ for some $m\in\bz$. Then the adjunction ensures that $m\langle w,w\rangle=2\langle w,x\rangle$. One can therefore obtain
$$(\funct{T}_{\funct{F}})_*x=-(\funct{F}\funct{R}-\funct{Id}_S)_*x=x-\frac{2\langle x,w\rangle}{\langle w,w\rangle} w$$
for any class $x\in\Num(S)$. It is exactly the reflection at $w$. One can see that it is orientation-reversing by $\langle w,w\rangle=-2\chi_{\sche}(v,v)>0$, contradicting Theorem \ref{2_K3-criterion-we-need}.

In fact, the orthogonal subspace $w^{\perp}$ in $\Num(S)_{\br}$ contains a vector $u$ satisfying $\langle u,u\rangle>0$. Since $\langle w,w\rangle>0$, the plane $\br w\oplus \br u$ is positive definite. It is direct to check the reflection $(\funct{T}_{\funct{F}})_*$ reverses its orientation.
\end{proof}
\end{proposition}

Also, we need the following observation to exclude non-trivial admissible subcategories with trivial numerical Grothendieck group.

\begin{proposition}\label{4_quasi-En-rank0}
Consider a quasi-Enriques category $\sche$ and the functor $\funct{F}$ defining it, such that $\Hoch^0(\sche)$ has no non-trivial idempotent and $\funct{T}_{\funct{F}}\funct{F}\simeq\funct{F}[2m]$ for some integer $m$. Then one has $\funct{S}_{\sche}\simeq[2m+1]$.

\begin{proof}
One can see $\funct{F}\funct{S}_{\sche}[-1]\simeq\funct{T}_{\funct{F}}\funct{F}\simeq\funct{F}[2m]$ by assumption and Proposition \ref{4_quasi-En-ppst}. 

Applying the right adjoint $\funct{R}$ of $\funct{F}$ to $\funct{F}\funct{S}_{\sche}[-1]\simeq\funct{F}[2m]$ shows that 
$$
\funct{S}_{\sche}[-1]\oplus\funct{S}^2_{\sche}[-3]\simeq[2m]\oplus\funct{S}_{\sche}[2m-2]
$$
where one also uses $\funct{R}\funct{F}\simeq\funct{Id}_{\sche}\oplus\funct{S}_{\sche}[-2]$.

By Proposition \ref{3_observation-2}, one has either $\funct{S}_{\sche}[-1]\simeq[2m]$ or $\funct{S}_{\sche}[-1]\simeq\funct{S}_{\sche}[2m-2]$. In the first case, we conclude. In the latter case, one has $\funct{Id}_{\sche}\simeq[2m-1]$. Then
$$\Hom(E,E[2mn-n])=\Hom(E,E)\neq0$$ 
for any nonzero object $E$ of $\sche$ and $n\in\bz$, violating the properness of $\sche$.
\end{proof}
\end{proposition}

\subsection{Applications to Enriques categories}
In this subsection, we will fix an Enriques category $\sche$ over a 2-Calabi-Yau category $\schc$ and choose a generator $\Pi$ of the $\Stwo$-action on $\schc$. Then $\Pi$ is a non-trivial involution of $\schc$.

\begin{proposition}\label{4_En-is-quasi-En}
The forgetful functor $\funct{Forg}\colon\sche\rightarrow\schc$ realizes the Enriques category $\sche$ as a quasi-Enriques category. The associated spherical twist is $\Pi[1]$.

\begin{proof}
The inflation functor $\funct{Inf}\colon\schc\rightarrow\sche$ of the equivariant structure is both left and right adjoint to the forgetful functor $\funct{Forg}$. Also, one has $$\funct{Inf}\circ\funct{Forg}\simeq\funct{Id}_{\sche}\oplus\funct{S}_{\sche}[-2]$$ 
where we use the fact $\amalg=\funct{S}_{\sche}[-2]$. The adjunction unit is the inclusion of the first summand in this direct sum by construction. The statement about the associated spherical twist follows from the definitions.
\end{proof}
\end{proposition}

\begin{remark}
In fact, for an arbitrary equivariant category $\schx_G$, one can check that the forgetful functor $\schx_G\rightarrow\schx$ is a $\bp^n$-functor in the sense of \cite[Section 4]{Add16} if $G$ is a cyclic group of order $n+1$.
\end{remark}

\begin{remark}\label{4_En-is-quasi-En-rmk}
Consequently, the composite $\funct{F}:=\funct{Forg}\circ\funct{J}$ realizes an admissible subcategory $\funct{J}\colon\schr\hookrightarrow\sche$ as a quasi-Enriques category. In this case, one has $\Pi\funct{F}\simeq\funct{F}$.
\end{remark}

It immediately leads to the following assertion.

\begin{corollary}\label{4_En-rank1-exclusion}
A noncommutative variety $\schr$ satisfying the following conditions
\begin{itemize}
	\item[(a)] $\rank\Num(\schr)=1$;
	\item[(b)] a generator $v$ of $\Num(\schr)$ satisfies $\chi_{\schr}(v,v)<0$;
\end{itemize}
cannot be an admissible subcategory of an Enriques category over a K3 surface.

\begin{proof}
It follows from Proposition \ref{4_quasi-En-subcat}, Proposition \ref{4_quasi-En-rank1} and Proposition \ref{4_En-is-quasi-En}.
\end{proof}
\end{corollary}

\section{The applications}
In this section, we will consider a semiorthogonal decomposition
\begin{equation}\tag{$\star$}\label{4_label}
	\schx=\langle\schr,\schs\rangle
\end{equation}
where the noncommutative variety $\schx$ is the Kuznetsov component of a quartic double solid or a special Gushel--Mukai threefold. It is worth mentioning that these two categories are deformation equivalent by \cite[Theorem 1.4]{KS25}.

\subsection{Some numerical invariants}
The main numerical invariants of the noncommutative variety $\schx$ are computed in \cite{Ku09,Ku09b,KP18}, see also \cite{Liu26a,PPZ26}.

\begin{proposition}\label{5_numerical}
The numerical Grothendieck group $\Num(\schx)$ has rank two and for a suitable basis of $\Num(\schx)$ its intersection matrix is
$$\begin{pmatrix}
	-1&0\\
	0&-1
\end{pmatrix}$$
under the Euler pairing. Also, one has $\Hoch_{\bullet}(\schx)=\bc^{10}[1]\oplus\bc^2[0]\oplus\bc^{10}[-1]$.
\end{proposition}

One can assume that $\rank\Num(\schr)\leq\rank\Num(\schs)$ after a mutation. Then the semi-orthogonal decomposition \eqref{4_label} is said to have \emph{type $(1,1)$} if the two sublattices have the same rank and is said to have \emph{type $(0,2)$} otherwise.

\begin{proposition}\label{5_Hochschild-estimation}
Consider the semiorthogonal decomposition \eqref{4_label}, then
\begin{enumerate}
	\item $\dim\Hoch_0(\schr)=\dim\Hoch_0(\schs)=1$ if \eqref{4_label} has type $(1,1)$;
	\item $\Hoch_0(\schr)=0$ and $\Hoch_*(\schr)$ is supported in degree $\pm1$ if \eqref{4_label} has type $(0,2)$.
\end{enumerate}
	
\begin{proof}
One just applies Proposition \ref{3_observation-3} and uses the definitions.
\end{proof}
\end{proposition}

\subsection{The conclusion}\label{5_ss-conclusion}
Now we can conclude Theorem \ref{1_main}.

\begin{proposition}\label{5_main-rank1}
The semiorthogonal decomposition \eqref{4_label} cannot have type $(1,1)$.

\begin{proof}
Otherwise, the admissible subcategory $\schr$ satisfies (a) in Corollary \ref{4_En-rank1-exclusion}. One can see that it also satisfies (b) according to Proposition \ref{5_numerical}. In this case, one can apply Corollary \ref{4_En-rank1-exclusion} to find a contradiction.
\end{proof}
\end{proposition}

\begin{remark}
This argument applies to some admissible subcategories of Enriques categories, such as Kuznetsov components of Enriques surfaces \cite{LNSZ21,LSZ22}. 
\end{remark}

So the semiorthogonal decomposition \eqref{4_label} can only have type $(0,2)$. To conclude the proof of Theorem \ref{1_main}, it suffices to prove the following statement.

\begin{proposition}\label{5_main-rank0}
The semiorthogonal decomposition \eqref{4_label} has type $(0,2)$ if and only if the admissible subcategory $\schr$ is the zero category.
	
\begin{proof}
Otherwise, one can assume that $\Hoch^0(\schr)$ has no non-trivial idempotent after passing to a non-trivial admissible subcategory using Proposition \ref{3_observation-1}.

At first, one has $\Hoch_0(\schr)=\Num(\schr)=0$ and $\Hoch_{\bullet}(\schr)$ is supported in degree $\pm1$ by Proposition \ref{5_Hochschild-estimation}. The spherical functor $\funct{F}$ in Remark \ref{4_En-is-quasi-En-rmk} satisfies $\funct{F}_*=0$ since the category $\cate{D}^b(S)$ has trivial odd Hochschild homology. It follows from the construction that the twist $\funct{T}_{\funct{F}}$ on $\cate{D}^b(S)$ induces the trivial Hodge isometry.

In either case, one claims that $\funct{T}_{\funct{F}}$ and $\Pi$ satisfy the conditions in Proposition \ref{2_K3-property-Bridgeland-coro}, which ensures that $\funct{T}_{\funct{F}}\simeq[2m]$. At first, one has $\funct{T}_{\Pi\funct{F}}\Pi\simeq\Pi\funct{T}_{\funct{F}}$ by construction of spherical functors. It follows that $\funct{T}_{\funct{F}}\Pi\simeq\Pi\funct{T}_{\funct{F}}$ as $\Pi\funct{F}\simeq\funct{F}$, see Remark \ref{4_En-is-quasi-En-rmk}. So it remains to describe the set $\Fix(\Pi):=\{\sigma\in\Stab(S)\mid\Pi.\sigma=\sigma\}$. It contains some $\sigma_0$ in $\Stab^{\dag}(S)$ according to \cite[Lemma 6.3 and Lemma 7.2]{PPZ26}. Any element $\tau$ in $\Fix(\Pi)$ induces an $\funct{S}_{\schx}[-2]$-fixed stability condition $\tau_{\Pi}$ on $\schx$ by Example \ref{3_invariant-stab-Enriques-Example}. Any two induced stability conditions on $\schx$ are in the same $\grp$-orbit by virtue of Theorem \ref{2_Serre-invariant}. Taking the residual action, the induced stability condition $\tau_{\Pi}$ goes back to $\tau$ up to a $\grp$-action according to Proposition \ref{3_invariant-stab-go-back}. So one can conclude that $\Fix(\Pi)=\sigma_0.\grp$ using \cite[Proposition 3.13]{Liu26a}. 

Consequently, one has $\funct{T}_{\funct{F}}\cong[2m]$ for some integer $m$. Then $\funct{S}_{\schr}\simeq[2m+1]$ using Proposition \ref{4_quasi-En-rank0}. It follows that $\Hoch_{2m+1}(\schr)\simeq\Hoch^0(\schr)\neq0$. So one has $m\in\{0,-1\}$ because $\Hoch_{\bullet}(\schr)$ is supported in degree $\pm1$. 

Now we claim that $\funct{F}\funct{R}\simeq\funct{Id}_S\oplus[2m-1]$. Here one considers the triangle $$\funct{F}\funct{R}\rightarrow\funct{Id}_S\stackrel{\alpha}{\rightarrow}[2m]\rightarrow\funct{F}\funct{R}[1]$$ 
of dg functors and it suffices to show $\alpha=0$. It is automatic when $m=-1$. If $m=0$, then $\alpha\in\Hoch^0(S)=\bc\cdot\funct{Id}_S$ and $\alpha\neq0$ would imply $\funct{F}\funct{R}=0$. However, it will lead to $\funct{F}\oplus\funct{F}\funct{S}_{\schr}[-2]=0$ and then $\funct{Id}_{\schr}\oplus\funct{S}_{\schr}[-2]\simeq\funct{R}\funct{F}=0$, impossible.

The construction in Remark \ref{4_En-is-quasi-En-rmk} ensures $\Pi\funct{F}\simeq\funct{F}$. So one has $$\Pi\oplus\Pi[2m-1]\simeq\Pi\funct{F}\funct{R}\simeq\funct{F}\funct{R}\simeq\funct{Id}_S\oplus[2m-1]$$
which is impossible by Proposition \ref{3_observation-2} as $\Pi$ is a non-trivial involution.
\end{proof}
\end{proposition}

In conclusion, we exclude any possible non-trivial decomposition \eqref{4_label}.

\begin{theorem}[Theorem \ref{1_main}]\label{5_main}
The Enriques category $\schx$ is indecomposable.

\begin{proof}
One combines Proposition \ref{5_main-rank1} and Proposition \ref{5_main-rank0}.
\end{proof}
\end{theorem}

\subsection{Some further observations}\label{SS_4.3}
The two Enriques categories discussed in this note can be seen as special examples of a large family of Enriques categories over K3 surfaces introduced in \cite{Liu26b}. Here we include a brief introduction.

Example \ref{2_EC-example1} is well-defined for any quartic K3 surface, which belongs to a very general family of K3 surfaces of degree $4$. Similarly, Example \ref{2_EC-example2} is defined for any K3 surface in a family of very general K3 surfaces of degree $10$. 

This phenomenon is generalized. It is shown in \cite{Liu26b} that, for any positive integer $t$ satisfying some arithmetic conditions, one can find a very general family of polarized K3 surfaces with degree $2t$, such that there exists an Enriques category deformation equivalent to Example \ref{2_EC-example1} and Example \ref{2_EC-example2} over any member in the family. 

Any polarized K3 surface in such a very general family is called a \emph{Bridgeland--Enriques general} K3 surface. The key properties are the following.

\begin{theorem}\label{5_BE-properties}
Consider a Bridgeland--Enriques general K3 surface $S$ of degree $2t$ for some $t\geq 2$, then there exists an Enriques category $\sche$ over $S$ such that 
\begin{enumerate}
	\item the involution $\Pi$ on $\cate{D}^b(S)$ generating the dual group action on $\cate{D}^b(S)$ fixes a stability condition $\sigma_0$ in $\Stab^{\dag}(S)$;
	\item the $\Pi$-invariant sublattice of $\Num(S)$ has rank two with intersection matrix
	$$\begin{pmatrix}
		-2&0\\
		0&-2
	\end{pmatrix}$$
	under the Euler pairing $\chi_S$;
	\item $\sche$ is a noncommutative deformation of the Kuznetsov component of a quartic double solid or a special Gushel--Mukai threefold.
\end{enumerate}

\begin{proof}
The assertion (1) is a part of \cite[Definition 1.1]{Liu26b}. The assertion (2) follows from \cite[Proposition 3.11]{Liu26b}. The assertion (3) is a combination of \cite[Theorem 1.5 and Theorem 1.12]{Liu26b} and the fact that Kuznetsov components of Gushel--Mukai varieties with dimension three or five are deformation equivalent (see \cite[Lemma 3.8 and Proposition A.2]{KP18} and \cite[Theorem 1.6]{KP23}).
\end{proof}
\end{theorem}

Then one can generalize Proposition \ref{5_numerical} as follows.

\begin{proposition}\label{5_BE-corollary1}
Consider an Enriques category $\sche$ in Theorem \ref{5_BE-properties}, then one has $\Hoch_{\bullet}(\sche)=\bc^{10}[1]\oplus\bc^2[0]\oplus\bc^{10}[-1]$. Also, the numerical Grothendieck group $\Num(\sche)$ has rank two and the Euler pairing on it is symmetric and negative definite.

\begin{proof}
The statement about Hochschild homology follows from Proposition \ref{5_numerical} and \cite[Theorem 3.5]{Per22}. Then one concludes by Proposition \ref{3_observation-3} and Theorem \ref{5_BE-properties}(2).
\end{proof}
\end{proposition}

Using this, one can prove the following analogues of Example \ref{3_invariant-stab-Enriques-Example}.

\begin{proposition}\label{5_BE-corollary2}
Consider an Enriques category $\sche$ in Theorem \ref{5_BE-properties}, then
\begin{enumerate}
	\item the forgetful functor $\sche=\cate{D}^b(S)_{\Stwo^\vee}\rightarrow\cate{D}^b(S)$ satisfies Proposition \ref{3_invariant-stab-general};
	\item the Euler form of $\cate{D}^b(S)$ is positive definite on $\ker(f_{\br})$ for the map $f\colon\Num(S)\rightarrow\Num(\sche)$ induced by the forgetful functor $\cate{D}^b(S)=\sche_{\Stwo}\rightarrow\sche$.
\end{enumerate}

\begin{proof}
(1) It suffices to verify that $\amalg:=\funct{S}_{\sche}[-2]$ induces the trivial action on $\Num(\sche)$ by Remark \ref{3_invariant-stab-general-rem}. Proposition \ref{5_BE-corollary1} implies that the form $\chi_{\sche}$ is non-degenerate and symmetric on $\Num(\sche)$, hence $(\funct{S}_{\sche})_*=\id_{\Num(\sche)}$ by Serre duality. So $\amalg_*$ is trivial, too.

(2) The inflation functor $\funct{Inf}\colon\sche\rightarrow\cate{D}^b(S)$ induces a homomorphism $i\colon \Num(\sche)\rightarrow\Num(S)$. The adjunction between inflation and forgetful functor provides
$$\chi_S(i_{\br}(x),y)=\chi_{\sche}(x,f_{\br}(y))$$
so $\ker(f_{\br})$ is the orthogonal of $\image(i_{\br})$ in $\Num(S)_{\br}$ with respect to $\chi_S$, because $\chi_{\sche}$ is non-degenerate (Proposition \ref{5_BE-corollary1}). The Euler pairing on $\Num(S)_{\br}$ has signature $(\rho,2)$ where $\rho$ is the Picard rank of $S$, so it remains to show that $\image(i_{\br})$ is negative definite with respect to $\chi_S$. This follows from adjunction and Proposition \ref{5_BE-corollary1}.
\end{proof}
\end{proposition}

In this case, the argument for Theorem \ref{1_main} almost works verbatim for $\sche$. The only missing ingredient is Theorem \ref{2_Serre-invariant}.

\begin{corollary}
Consider an Enriques category $\sche$ in Theorem \ref{5_BE-properties} and assume that all the $\funct{S}_{\sche}$-invariant stability conditions on $\sche$ are in the same $\grp$-orbit of the complex manifold $\Stab(\sche)$, then $\sche$ is indecomposable.

\begin{proof}
One applies the argument for Theorem \ref{1_main} to the Enriques category $\sche$ by using Theorem \ref{5_BE-properties}(1), Proposition \ref{5_BE-corollary1}, and Proposition \ref{5_BE-corollary2}.
\end{proof}
\end{corollary}

Alternatively, one can prove indecomposability of the Enriques category $\sche$ by assuming a strong Bridgeland conjecture (see \cite[Conjecture 1.2]{Bri08} and \cite{BB17}). It is used to replace the function of Proposition \ref{2_K3-property-Bridgeland-coro} in Proposition \ref{5_main-rank0}.

\begin{corollary}
Consider an Enriques category $\sche$ in Theorem \ref{5_BE-properties} and the associated K3 surface $S$. Assume that $\Stab^{\dag}(S)$ is contractible and $\Aut(\cate{D}^b(S))=\Aut^{\dag}(\cate{D}^b(S))$, then $\sche$ is indecomposable.
	
\begin{proof}
It remains to exclude non-trivial admissible subcategories with rank zero numerical Grothendieck groups. By the proof of Proposition \ref{5_main-rank0}, the only missing ingredient is a replacement of Proposition \ref{2_K3-property-Bridgeland-coro}.
	
Using \cite[Theorem 7]{Av14} and the contractibility of $\Stab^{\dag}(S)$, the subspace
$$\Fix^{\dag}(\Pi):=\{\sigma\in\Stab^{\dag}(S)\mid \Pi.\sigma=\sigma\}$$
is connected if it is non-empty. The non-emptiness is Theorem \ref{5_BE-properties}(1) and we choose a point $\sigma_0$ in it. Then we claim that $\Fix^{\dag}(\Pi)=\sigma_0.\grp$ and it suffices to show that every $\grp$-orbit in $\Fix^{\dag}(\Pi)$ is open.

According to \cite[Corollary 1.3]{Bri07}, the projection
$$\Stab^{\dag}(S)\rightarrow\Hom(\Num(S),\bc),\quad \sigma=(\schp,Z)\mapsto Z$$
is a local homeomorphism onto its image. It is $\Pi$-equivariant and its restriction to $\Fix^{\dag}(\Pi)$ is locally a homeomorphism onto the $\Pi_*$-fixed locus, so it is sufficient to work on the space of central charges. 

We identify $\Num(S)_{\bc}$ with $\Hom(\Num(S),\bc)$ using the Mukai pairing. The $\Pi$-invariant subspace $(\Num(S)_{\br})_{\Pi_*}$ of $\Num(S)_{\br}$ is a positive definite real plane with respect to the Mukai pairing by Theorem \ref{5_BE-properties}(2). So, for any $\sigma=(\schp,Z)\in\Fix^{\dag}(\Pi)$, the real and imaginary parts of $Z$ form an oriented basis of $(\Num(S)_{\br})_{\Pi_*}$. The group $\GLp(2,\br)$ acts transitively on the ordered bases of this plane having the same orientation, so nearby $\Pi$-invariant central charges differ from $Z$ simply by an orientation-preserving change of coordinates on $\bc$. It means that the $\GLp(2,\br)$-orbit of $Z$ is open and, consequently, the orbit $\grp.\sigma$ is open in $\Fix^{\dag}(\Pi)$.

In this case, one concludes that $\Fix^{\dag}(\Pi)=\sigma_0.\grp$. Then we can apply the argument of Proposition \ref{2_K3-property-Bridgeland-coro} to run Proposition \ref{5_main-rank0} as any autoequivalence $\Phi$ on $\cate{D}^b(S)$ preserves $\Stab^{\dag}(S)$ by our assumption.
\end{proof}
\end{corollary}

\section*{Acknowledgments}
The author would like to thank his advisors Laura Pertusi and Paolo Stellari for their support and helpful comments, and thank Sasha Kuznetsov for his suggestions on a preliminary version, which greatly improved the readability of the note.

\section*{AI Disclosure}
This note was solely written by the author and AI was used to assist editorial corrections. Before submitting this note to arXiv, the author was informed that AI could generate a substantially simpler proof, similar to the one presented here.

\end{document}